\documentclass[12pt]{amsart}
\usepackage{graphicx}
\usepackage[active]{srcltx}
\usepackage{amsthm,mathrsfs}
\usepackage{a4wide}
\usepackage{amsfonts}
\usepackage{amsmath}
\usepackage{amssymb}
\newtheorem{lemma}{Lemma}
\newtheorem{theorem}{Theorem}
\newtheorem{corollary}{Corollary}
\newtheorem{remark}{Remark}
\newtheorem{proposition}{Proposition}

\usepackage{float}
\usepackage{url}
\usepackage{enumitem}
\usepackage{epstopdf}

\usepackage{ mathrsfs }

\usepackage{calligra}
\usepackage{dsfont}

\makeatletter\def\blfootnote{\xdef\@thefnmark{}\@footnotetext}\makeatother
\allowdisplaybreaks
\title[The exact order of discrepancy for Levin's normal number in base 2]{\bf The exact order of discrepancy for Levin's normal number in base 2}

\author{Roswitha Hofer and Gerhard Larcher} 
\address{Institute of Financial Mathematics and Applied Number Theory, Johannes Kepler University Linz, Altenbergerstra{\ss}e 69, 4040 Linz, AUSTRIA}
\email{roswitha.hofer@jku.at, gerhard.larcher@jku.at}

\thanks{The authors are supported by the Austrian Science Fund (FWF): Project F5505-N26 and Project F5507-N26, which are both part of the Special Research Program ``Quasi-Monte Carlo Methods: Theory and Applications''}

\begin{document}

\begin{abstract}
Mordechay B. Levin in \cite{Lev} has constructed a number $\alpha$ which is normal in base 2, and such that the sequence $\left\{2^n \alpha\right\}_{n=0,1,2,\ldots}$ has very small discrepancy $D_N$. Indeed we have $N \cdot D_N = \mathcal{O} \left(\left(\log N\right)^2\right)$. That means, that $\alpha$ is normal of extremely high quality. In this paper we show that this estimate is best possible, i.e., $N\cdot D_N \geq c \cdot \left(\log N\right)^2$ for infinitely many $N$. 
\end{abstract}

\date{}
\maketitle

\section{Introduction and statement of the result} \label{sect_a}

A real number $\alpha \in [0,1)$ is called ``normal in base 2'' if in its base 2 representation $\alpha = 0. \alpha_1, \alpha_2 \ldots$ the following holds: For every positive integer $k$ and any 0-1 block $a_1 a_2 \ldots a_k \in \left\{0,1\right\}^k$ of length $k$ we have
$$
\lim_{N \rightarrow \infty} \frac{1}{N} \# \left\{0 \leq n < N \left|\right. \alpha_{n+1} \alpha_{n+2} \ldots \alpha_{n+k} = a_1 a_2 \ldots a_k\right\} = \frac{1}{2^k}.
$$
Of course this is equivalent with the following, seemingly more general, property: For any two blocks $a_1 \ldots a_k$ and $b_1 \ldots b_k$ in $\left\{0,1\right\}^k$ we say that $a_1 \ldots a_k \prec b_1 \ldots b_k$ iff $0 . a_1 \ldots a_k < 0 . b_1 \ldots b_k$. (resp. $a_1 \ldots a_k \preceq b_1 \ldots b_k$ iff $0 . a_1 \ldots a_k \leq 0 . b_1 \ldots b_k$). Then
\begin{eqnarray} \label{equ_a}
&&\lim_{N \rightarrow \infty} \frac{1}{N} \# \left\{0 \leq n < N \left| \right. a_1a_2\ldots a_k \preceq \alpha_{n+1} \ldots \alpha_{n+k} \prec b_1 \ldots b_k\right\} = \\
&&= 0 . b_1 \ldots b_k - 0 . a_1 \ldots a_k. \nonumber
\end{eqnarray}
It is an easy exercise to show that $\alpha$ is normal in base 2 iff the sequence $\left\{2^n \alpha\right\}_{n=0,1,\ldots}$ is uniformly distributed in $[0,1)$. That means: For any $a,b$ with $0 \leq a < b \leq 1$ we have
$$
\lim_{N \rightarrow \infty} \frac{1}{N} \# \left\{0 \leq n < N \left|\right. a \leq \left\{2^n \alpha\right\} < b \right\} = b-a.
$$
The ``quality'' of the uniform distribution of a sequence $(x_n)_{n=0,1,\ldots}$ in $[0,1)$ usually is measured with its discrepancy $D_N$. Here
$$
D_N:=\sup_{0 \leq a < b \leq 1} \left| \left. \frac{1}{N} \# \Big\{0 \leq n < N \right| a \leq x_n < b\Big\} - (b-a)\right|.
$$
$(x_n)_{n=0,1,\ldots}$ is uniformly distributed in $[0,1)$ iff $\lim \limits_{N \rightarrow \infty} D_N = 0$. \\
Now we have the following obvious relation between the discrepancy $D_N$ of $(\left\{2^n \alpha\right\})_{n = 0,1 \ldots}$ and the speed of convergence in \eqref{equ_a}: We have $a_1 a_2 \ldots a_k \preceq \alpha_{n+1} \ldots \alpha_{n+k} \prec b_1 \ldots b_k$ iff $0 . a_1 \ldots a_k \leq \left\{2^n \alpha\right\} < 0 . b_1 \ldots b_k$. Therefore
$$\left|\frac{1}{N} \# \left\{0\leq n < N\left|\right. a_1a_2 \ldots a_k \preceq \alpha_{n+1} \ldots \alpha_{n+k} \prec b_1 \ldots b_k\right\}-\Big(0 . b_1\ldots b_k - 0 . a_1 \ldots a_k\Big)\right|=$$
$$
= \bigg|\frac{1}{N} \# \Big\{0 \leq n < N \left|\right. 0 . a_1 \ldots a_k \leq \left\{2^n \alpha\right\}< 0 . b_1 \ldots b_k\Big\} -
$$
$$
- \Big(0 . b_1 \ldots b_k -0 . a_1 \ldots a_k\Big)\bigg| \leq D_N.
$$
On the other hand: Let $F(N)$ be such that for all positive integers $k$ and all blocks $a_1\ldots a_k$ and $b_1 \ldots b_k \in \left\{0,1\right\}^k$ we have
\begin{eqnarray} \label{equ_b}
&&\bigg|\frac{1}{N} \# \Big\{0 \leq n < N\left|\right. a_1 \ldots a_k \preceq \alpha_{n+1} \ldots \alpha_{n+k} \prec b_1 \ldots b_k \Big\}  - \\
&& - \Big(0 . b_1 \ldots b_k - 0 . a_1 \ldots a_k\Big)\bigg|\leq F(N). \nonumber
\end{eqnarray}
Let $a,b$ with $0 \leq a < b < 1$ be arbitrary and $\varepsilon > 0$. Let $k$ be such that $\frac{1}{2^k} < \varepsilon$ and $a_1 \ldots a_k$ and $b_1 \ldots b_k$ be such that $0 . a_1 \ldots a_k \leq a < 0 . a_1 \ldots a_k + \frac{1}{2^k}$ and $0 . b_1 \ldots b_k \leq b < 0 . b_1 \ldots b_k + \frac{1}{2^k}$. Further we denote by $\overline{b_1 \ldots b_k}$ resp. $\overline{a_1 \ldots a_k}$ the block of digits representing $0 . b_1 \ldots b_k + \frac{1}{2^k}$ resp. $0 . a_1 \ldots a_k + \frac{1}{2^k}$. Then

$$
\left|\frac{1}{N} \# \left\{0 \leq n < N \left|\right.a \leq \left\{2^n \alpha\right\} < b\right\} -(b-a)\right|\leq
$$
$$
\leq \max \left(\frac{1}{N} \# \left\{0\leq n < N\left|a_1 \ldots a_k \preceq \alpha_{n+1} \ldots \alpha_{n+k} \prec \overline{b_1 \ldots b_k}\right\} - \right.\right.
$$
$$
- \left(0 . \overline{b_1 \ldots b_k} - 0 . a_1 \ldots a_k\right) + \frac{1}{2^k},
$$
$$
\left(0 . b_1 \ldots b_k - 0 . \overline{a_1 \ldots a_k}\right) + \frac{1}{2^k}~-
$$
$$
-~ \# \left\{0 \leq n < N \left|~\overline{a_1 \ldots a_k} \preceq \alpha_{n+1} \ldots \alpha_{n+k}\prec b_1 \ldots b_k\right.\right\}\biggr)
$$
$$
\leq F(N) + \varepsilon.
$$
Hence the discrepancy $D_N$ of the sequence $\left\{2^n \alpha\right\}_{n = 0,1, \ldots}$ also is a perfect measure for the ``quality of the normality of $\alpha$ in base 2''. We will say: ``$D_N$ is the discrepancy of the normal number $\alpha$ in base 2''.\\

It was shown by W.M. Schmidt \cite{Schm}, and it is a well-known fact that there is a positive constant $c$, such that for every sequence $(x_n)_{n = 0,1, \ldots}$ in $[0,1)$ we have
$$
D_N\geq c \cdot \frac{\log N}{N}
$$
for infinitely many $N$. So also the discrepancy of any normal number $\alpha$ in base 2 is at least of order $c \cdot \frac{\log N}{N}$. This fact follows from the (highly non-trivial!) general result of Schmidt, but it can also be deduced rather easily directly by the following simple argument:\\

Assume that $D_N \leq \frac{\log N}{N}$ holds for all $N$. Let $M \in \mathbb{N}, L:= \left\lfloor\frac{\log M}{2 \log 2}\right\rfloor$ and $U:= M + L$.\\
Then $\frac{L}{4} \geq \frac{1}{16 \log 2} \log U$ for $M$ large enough. We have
$$
\# \left\{0 \leq n < M \left|x_n \in \left[0,\frac{1}{2^L}\right)\right.\right\}\geq M \cdot \frac{1}{2^L}-M\cdot D_M \geq M \cdot \frac{1}{2^L} - \log M > 0
$$
for $M$ large enough.\\

Hence there is an $m$ with $0 \leq m < M$ and $x_m \in \left[0, \frac{1}{2^L}\right)$. Therefore $x_m, x_{m+1}, \ldots, x_{m+L-1} \in \left[0,\frac{1}{2}\right)$. Hence either
\begin{eqnarray*}
& \# & \left\{0 \leq n < m \left| x_n \in \left[0,\frac{1}{2}\right) \right.\right\} \leq \frac{m}{n} - \frac{L}{4} \quad\quad\quad~\text{or} \\
& \# & \left\{0 \leq n < m + L \left|x_n \in \left[0,\frac{1}{2}\right)\right.\right\} \geq \frac{m}{2} + \frac{3L}{4} = \frac{m+L}{2} + \frac{L}{4}
\end{eqnarray*}
Therefore there is an $N \leq U$ with $N D_N \geq \frac{L}{4} \geq \frac{1}{16 \log 2} \log U \geq \frac{1}{16 \log 2} \log N$. Of course with growing $M$ and hence growing $L$ we can prove the existence of infinitely many such $N$.\\

By an ingenious construction Mordechay B. Levin in \cite{Lev} provided a number $\alpha$ normal in base 2 with discrepancy $D_N \leq c\cdot \frac{(\log N)^2}{N}$ (with an absolute constant $c$). Until then it was only known that for almost all $\alpha$ we have $D_N = \mathcal{O} \left(\left(\frac{\log\log N}{N}\right)^{\frac{1}{2}}\right)$, see \cite{Gal}, and Korobov has given an explicit example of $\alpha$ with $D_N = \mathcal{O} \left(\left(\frac{1}{N}\right)^{\frac{1}{2}}\right)$. See \cite{Koro}. The most prominent normal number, the Champernowne number $\alpha$ is of rather bad quality. We have
$$
D_N \geq c\cdot \frac{1}{\log N}
$$
for infinitely many $N$. See for example \cite{Schiff}.\\

Nevertheless there still is a gap of one $\log N$-factor between the best known example of Levin and the currently best known lower bound for $D_N$. So the  main and certainly challenging question is, if either the upper or the lower bound (or both bounds) for the discrepancy of normal numbers can be improved. The first idea in an attempt to improve the upper bound could be to try to improve the discrepancy estimate given by Levin for his normal number $\alpha$. The aim of this paper is, to show that this attempt has to fail, since we will show
\begin{theorem}
Let $\alpha$ be Levin's normal number in base 2. (For the exact definition of $\alpha$ see Section~\ref{sect_b}.) Let $D_N$ be the discrepancy of the sequence $(\left\{2^n \alpha\right\})_{n = 0, \ldots, N-1}$. Then there is a positive constant $c$ such that
$$
D_N \geq c \cdot \frac{\left(\log N\right)^2}{N}
$$
for infinitely many $N$.
\end{theorem}
So the \textbf{main question} remains open:\\
What is the best possible order of normality in base 2 i.e., what is the smallest possible order of the discrepancy $D_N$ of sequences of the form $\left\{2^n \alpha\right\}_{n = 0,1,\ldots}$. Is it $\frac{\log N}{N}$, or $\frac{\left(\log N\right)^2}{N}$, or something in between?

\section{Levin's normal number $\alpha$ and two auxiliary results} \label{sect_b}

Levin's normal number $\alpha$ in base 2 is defined as follows: We denote the representation of $\alpha$ in base 2 by\\

$\alpha = 0 . \underset{\mathcal{A}_1}{\underbrace{\alpha_1\alpha_2 \ldots \alpha_8}} ~ \underset{\mathcal{A}_2}{\underbrace{\alpha_9 \ldots \alpha_{72}}} \quad\ldots \quad \underset{\mathcal{A}_m}{\underbrace{\ldots\ldots}} \quad \ldots\ldots$ .\\

Here the blocks $\mathcal{A}_m$ consist of $2^m \cdot 2^{2^{m}}$ digits $\alpha_i$ for $m = 1,2, \ldots$.\\

We set $n_1 := 0$ and $n_m := 2^1 \cdot 2^{2^{1}} + 2^2 \cdot 2^{2^{2}} + \ldots + 2^{m-1} \cdot 2^{2^{m-1}}$ for $m = 2,3, \ldots$. Then block $\mathcal{A}_m$ starts with $\alpha_{n_m +1}$. The block $\mathcal{A}_m$ is of the form

$$
\underbrace{d_0 (0) \ldots d_k(0) \ldots d_{2^m-1}(0)} \quad \underbrace{d_0(1) \ldots d_k(1) \ldots d_{2^m-1}(1)} \quad \ldots
$$
$$
\ldots \quad \underbrace{d_0(n) \ldots d_k(n) \ldots d_{2^m-1}(n)} \quad\ldots \quad\underbrace{d_0\left(2^{2^{m}}-1\right) \quad \ldots \quad d_{2^m-1}\left(2^{2^{m}}-1\right)}.
$$

For $n$ between $0$ and $2^{2^{m}}-1$ we set $n:= e_0(n) + 2 \cdot e_1(n) + \ldots + 2^{2^{m}-1} \cdot e_{2^m-1} (n)$.\\

Then $d_k(n) := p_{k,0} e_0(n) + \ldots + p_{k,2^m-1}e_{2^m-1}(n) \mod2$, with $p_{i,j} := \binom{i+j}{j} \mod2$ for all non-negative integers $i$ and $j$. \\
We define the $\mathbb{N}_0 \times \mathbb{N}_0$ - matrix $P$ as
$$
P:= \left(p_{ij}\right)_{i, j=0,1,\ldots} = 
\renewcommand*{\arraystretch}{1.5}\begin{pmatrix} 
\binom{0}{0} & \binom{1}{1} & \binom{2}{2} & \binom{3}{3} & \ldots \\
\binom{1}{0} & \binom{2}{1} & \binom{3}{2} & \ldots \\
\binom{2}{0} & \binom{3}{1} & \ldots \\
\binom{3}{0} & \ldots \\
\vdots \end{pmatrix}\pmod{2}.
$$
Levin in Theorem 2 in \cite{Lev} has shown, that for this $\alpha$ for the discrepancy $D_N$ of the sequence $\left\{2^n \alpha\right\}_{n = 0,1,\ldots}$ we have $D_N = \mathcal{O} \left(\frac{\left(\log N\right)^2}{N}\right)$.\\

We will have to use this upper bound for $D_N$ also in our proof of our lower bound for $D_N$. Further we will need two auxiliary results.\\

First, we will use the second result (formula (55)) in Corollary 2 in \cite{Lev}. This is 
\begin{lemma} \label{lem_a}
For every $m$ and every $\gamma$ with $0 \leq \gamma < 1$ we have
$$
\# \left\{n_m \leq n < n_m +2^m \cdot 2^{2^{m}} \left|\left\{2^n \alpha\right\} \in [0,\gamma)\right\}\right. = \gamma \cdot 2^m \cdot 2^{2^{m}} + \varepsilon \cdot 2^m
$$
with some $\varepsilon$ with $\left|\varepsilon\right|< 5$. (This $\varepsilon$ here and in the following, is not a constant but denotes a variable with bounded absolute value!)
\end{lemma}
Further we will use the following sharper version of Lemma 5 in \cite{Lev}.

\begin{lemma} \label{lem_b}
For every positive integer $m$ we have: For every $0 \leq i < 2^m$, every integer $B$ with $0 \leq B < 2^{2^{m}-i}$, and for every integer $c$ with $0 \leq c < 2^{i}$ with the exception of at most $2^{m+1}$ such $c$ we have
$$
\mathcal{N} := \# \left\{0 \leq k < 2^m, B \cdot 2^{i} \leq n < B \cdot 2^{i} + 2^{i} \left|\left\{2^{n_m + {2^m} n+k} \cdot \alpha\right\}\in \left[\frac{c}{2^{i}}, \frac{c+1}{2^{i}}\right)\right\}\right. = 2^m.
$$
\end{lemma}
\begin{proof}
Each $n$ with $B \cdot 2^{i} \leq n < B \cdot 2^{i} + 2^{i}$ can be uniquely represented in the form
$$
n = e_0(n) + e_1(n) \cdot 2 + \ldots + e_{i-1}(n) \cdot 2^{i-1} + b_0 \cdot 2^{i} + b_1 \cdot 2^{i+1} + \ldots + b_{2^m -i-1} \cdot 2^{2^{m}-1}.
$$
Let $c:= c_0 + c_1 \cdot 2 + \ldots + c_{i-1} 2^{i-1}$. Fix a $k$ with $0 \leq k < 2^m$. Then $\left\{2^{n_m + {2^m} n+k} \alpha\right\} \in \left[\frac{c}{2^{i}}, \frac{c+1}{2^{i}}\right)$ is equivalent with\\

\textbf{Case 1:} If $k+i \leq 2^m$:
\begin{eqnarray} \label{equ_c}
&& d_k(n) = c_0 \\ \nonumber
&& d_{k+1}(n) = c_1 \\ \nonumber
&& \quad \vdots \\ \nonumber
&& d_{k+i-1}(n) = c_{i-1} 
\end{eqnarray}
respectively \\
\textbf{Case 2:} If $k+i > 2^m$:
\begin{eqnarray} \label{equ_d}
&& d_k(n) = c_0 \\ \nonumber
&& d_{k+1}(n) = c_1 \\ \nonumber
&& \quad \vdots \\ \nonumber
&& d_{2^m-1} (n) = c_{2^m-k-1} \\ \nonumber
&& d_0 (n+1) = c_{2^m-k} \\ \nonumber
&& \quad \vdots \\ \nonumber
&& d_{i+k-2^m-1} (n+1) = c_{i-1}.
\end{eqnarray}
Let us consider first Case 1. The system \eqref{equ_c} is equivalent with
$$
p_{k,0} e_0(n) + \ldots + p_{k, i-1} e_{i-1}(n) + \beta_k = c_0
$$
$$
\vdots
$$
$$
p_{k+i-1,0} e_0(n) + \ldots + p_{k+i-1, i-1} e_{i-1}(n) + \beta_{k+i-1} = c_{i-1}
$$
with $\beta_{k+j} := p_{k+j, i} \cdot b_0 + \ldots + p_{k+j, 2^m-1} \cdot b_{{2^m}-i-1}$. The matrix $\left(p_{k+j, u}\right)_{j, u = 0, \ldots, i-1}$ is regular (see Lemma 4 in \cite{Lev}), hence for each choice of $c$ and every such $k$ there is exactly one $n$ such that $\left\{2^{n_m+{2^m} n+k}\alpha\right\} \in \left[\frac{c}{2^{i}}, \frac{c+1}{2^{i}}\right)$.\\

Let us consider now Case 2. This case is more delicate since then the system contains now also variables $e_0(n+1), \ldots, e_{i-1}(n+1)$. It will turn out that this does not make any problem, but only in one case, namely if $n=2^{i}-1$. In this case we also have to take into account that then also the digit $e_i(n+1)=1$ will appear. But here once more we have to be careful: If $i=2^m$, then the ``$n+1$'' in this system is not equal to $2^{2^{m}}$ but it equals 0.\\

To handle the now relevant system (see below) we will make use of the following special form of the matrix $P$: \\
$P$ is generated by starting with the $1\times 1$-matrix $A_0 = (1)$, and then by successively carrying out the transformation $A_m \rightarrow A_{m+1} := \begin{pmatrix} A_m & A_m \\ A_m & 0 \end{pmatrix}.$ \\
I.e., 
$$
(1) \rightarrow \begin{pmatrix} 1 & 1 \\ 1 & 0 \end{pmatrix} \rightarrow \begin{pmatrix} 1 & 1 & 1 & 1 \\ 1 & 0 & 1 & 0 \\ 1 & 1 & 0 & 0  \\1 & 0 & 0 & 0 \end{pmatrix} \rightarrow \ldots .
$$
Hence the left upper $2^m \times 2^m$ - submatrix $A_m$ of $P$ is a left upper triangle matrix. Hence, for $0 \leq i, j < 2^m$ we have $p_{i,j} =0$ whenever $i+j \geq 2^m$. The system which we have to deal with, now is of the form:\\
$$
\left(\begin{tabular}{l|l} A & B \\ \hline C & D \end{tabular} \right) = \left(\begin{tabular}{l} $c_0$ \\ $\vdots$ \\ $c_{2^m-k-1}$ \\ \hline $c_{2^m-k}$ \\ \vdots \\ $c_{i-1}$ \end{tabular} \right)
$$
where
{\scriptsize{$$
A:= \left(\begin{tabular}{c} $p_{k,0} e_0(n) + \ldots + p_{k,2^m-k-1} e_{2^m-k-1}(n)$ \\ \vdots \\ $p_{2^m-1, 0}e_0(n) + \ldots + p_{2^m-1, 2^m-k-1}e_{2^m-k-1}(n)$ \end{tabular}\right),
$$
~\

$$
B:= \left(\begin{tabular}{c} $+~ p_{k,2^m-k} e_{2^m-k}(n) + \ldots + p_{k,i-1} e_{i-1} (n) + \beta_k$ \\ \vdots \\ $+~p_{2^m-1, 2^m-k} e_{2^m-k}(n) + \ldots + p_{2^m-1, i-1} e_{i-1}(n) + \beta_{2^m-1} $ \end{tabular}\right),
$$
~\

$$
C:= \left(\begin{tabular}{c} $p_{0,0}e_0 (n+1) + \ldots + p_{0,2^m-k-1}e_{2^m-k-1}(n+1)$ \\ \vdots \\ $p_{i+k-2^m-1,0}e_0 (n+1) + \ldots + p_{i+k+2^m-1, 2^m-k-1}e_{2^m-k-1}(n+1)$ \end{tabular}\right),
$$
~\

$$
D:= \left(\begin{tabular}{c} $+~p_{0,2^m-k}e_{2^m-k}(n+1) + \ldots + p_{0,i-1}e_{i-1}(n+1) + \beta_0 + \tau_0$ \\ \vdots \\ $+~p_{i+k-2^m-1, 2^m-k} e_{2^m-k}(n+1) + \ldots + p_{i+k-2^m-1, i-1} e_{i-1} (n+1) + \beta_{i+k-2^m-1} + \tau_{i+k-2^m-1}$\end{tabular}\right).
$$ }}

Here $\tau_j := \begin{cases} p_{j,i} &\text{if}~ n = 2^{i}-1 ~\text{with}~i \neq 2^m, \\ 0 &\text{otherwise.}\end{cases}$\\

By the property of $P$ pointed out above, we have $B\equiv 0$. The matrix defining part $A$ is regular (see Lemma 4 in \cite{Lev}). Hence $e_0(n), \ldots, e_{2^m-k-1}(n)$ are uniquely determined by the upper part of the system. Hence also part $C$ is determined. Say, 
$$
C:=\left(\begin{tabular}{l} $\delta_0$ \\ $\vdots$ \\ $\delta_{i+k-2^m-1}$ \end{tabular}\right).
$$
That is, we arrive at the system
$$
p_{0,2^m-k} e_{2^m-k}(n+1) + \ldots + p_{0,i-1}e_{i-1}(n+1) + \delta_0 + \beta_0 + \tau_0 = c_{2^m-k}
$$
$$
\vdots
$$
$$
p_{i+k-2^m-1, 2^m-k}e_{2^m-k}(n+1) + \ldots + p_{i+k-2^m-1, i-1} e_{i-1}(n+1) + \delta_{i+k-2^m-1} + 
$$
$$
+ \beta_{i+k-2^m-1} + \tau_{i+k-2^m-1} = c_{i-1}
$$
The sub-matrix of $P$ defining this system (again by Lemma 4 in \cite{Lev}) is regular.\\

Consider now the whole system $\left(\begin{tabular}{l|l} A & B \\ \hline C & D \end{tabular} \right) = c$ first without the entries $\tau_j$. 
As we have pointed out above, this system has a unique solution, say $(e_0,e_1,\ldots,e_{2^m-k-1},f_{2^m-k},\ldots,f_{i-1})$, where $e_0,e_1,\ldots,e_{2^m-k-1}$ are determined by the upper part of the system, and $e_j=e_j(n)$ for $j=0,1,\ldots,2^m-k-1$. 

If now (Case 2.1)
$$(e_0,e_1,\ldots,e_{2^m-k-1})\neq(1,1,\ldots,1),$$
then certainly $n\neq 2^i-1$, hence all $\tau_j=0$, and $f_j=e_j(n+1)=e_j(n)$ for $j=2^m-k,\ldots,i-1$ gives the unique solution $n$. 

If (Case 2.2) 
$$(e_0,e_1,\ldots,e_{2^m-k-1})=(1,1,\ldots,1)$$
and 
$$(f_{2^m-k},\ldots,f_{i-1})\neq(0,0,\ldots,0),$$
say
$$(e_0,e_1,\ldots,e_{2^m-k-1}|f_{2^m-k},\ldots,f_{i-1})=(1,1,\ldots,1|0,0,\ldots,0,1,\ldots),$$
then we set 
$$(e_0(n),e_1(n),\ldots,e_{2^m-k-1}(n)|e_{2^m-k}(n),\ldots,e_{i-1}(n)):=(1,1,\ldots,1|1,1,\ldots,1,0,\ldots).$$
The corresponding $n$ is different from $2^i-1$, hence $\tau_j=0$ for all $j$, and therefore this $n$ gives the unique solution of our system. 

Finally (Case 2.3), let $\tilde{c}$ be the unique integer such that the system (without the $\tau_j$) has the unique solution 
$$(e_0,e_1,\ldots,e_{2^m-k-1}|f_{2^m-k},\ldots,f_{i-1})=(1,1,\ldots,1|0,0,\ldots,0).$$
Only in this case, for this single $\tilde{c}$, it could happen that the system with the $\tau_j$ has no solution, i.e., that there is no element $\{2^{n_m+2^mn+k}\alpha\}$ in $\left[\left.\frac{\tilde{c}}{2^i},\frac{\tilde{c}+1}{2^i}\right)\right.$. Consequently also at most one of the intervals $\left[\left.\frac{c}{2^i},\frac{c+1}{2^i}\right)\right.$ with $c\neq \tilde{c}$ contains more than one of the points $\{2^{n_m+2^mn+k}\alpha\}$. This holds for every $k=0,1,\ldots,2^m-1$ and so the result follows. 
\end{proof}

\section{Some properties of the Pascal-matrix $P$} \label{sect_c}

We recall that the matrix $P$ (we will call it ``Pascal-matrix'') is of the form 
$$P= (p_{ij})_{i,j = 0,1,\ldots} = \left(\binom{i+j}{j} \mod 2\right)_{i,j = 0,1,\ldots}.$$

Let $m\in\mathbb{N}$. For fixed $t$ with $1 \leq t \leq 2^m$ and arbitrary $0 \leq k < 2^m$ let 
$$
A_{k,t} := \begin{pmatrix} p_{k, 0} &\ldots& p_{k,t-1} \\ \vdots & & \vdots \\ p_{k+t-1, 0} &\ldots& p_{k+t-1, t-1} \end{pmatrix}
$$
and
$$
B_{k,t} :=  \begin{pmatrix} p_{k, t} &\ldots &p_{k,2^m-1} \\ \vdots && \vdots \\ p_{k+t-1, t} &\ldots& p_{k+t-1, 2^m-1} \end{pmatrix}.
$$
By Lemma 4 in \cite{Lev} the matrix $A_{k,t}$ always is regular in $\mathbb{Z}_2$. \\

For given $k,t$ like above let $$c_{k+t,t}:= \left(\begin{tabular}{c} $p_{k+t, 0} \ldots \ldots p_{k+t,t-1}$ \end{tabular}\right),\quad\quad d_{k+t,t}:= \left(\begin{tabular}{c} $p_{k+t, t} \ldots \ldots p_{k+t,2^m-1}$ \end{tabular}\right)$$ and $$\xi_t := \left(\binom{t}{0}, \binom{t}{1}, \ldots, \binom{t}{t-1}\right)\pmod{2}.$$ 
Furthermore, we define $\overline{0}=(0,0,0,0,0,0,0,0)^T$ and $\overline{1}=(1,0,0,0,0,0,0,0)^T$.
See Figure \ref{fig:Figure1} for an illustration.\\


\begin{figure}[h]
	\centering
		\includegraphics[width=0.80\textwidth]{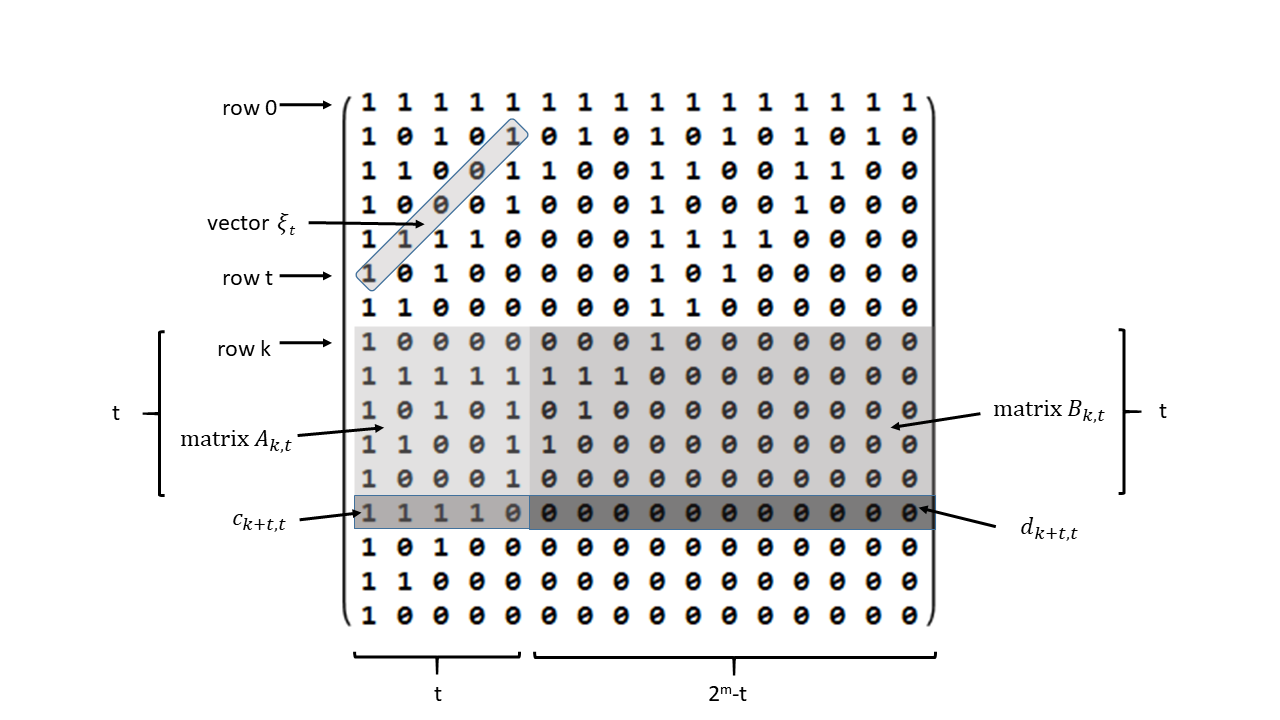}
	\caption{The magnitudes $\xi_t$, $A_{k,t}$, $B_{k,t}$, $c_{k+t,t}$, and $d_{k+t,t}$. }
	\label{fig:Figure1}
\end{figure}

\begin{lemma}\label{lem:1a}
Let $i$ and $u\in\mathbb{N}_0$. Then $\sum_{j=1}^u\binom{i+j}{j}\equiv\binom{i+1+u}{u}+1\pmod{2}$. 
\end{lemma}
\begin{proof}
This is an easy consequence of $\sum_{j=0}^u\binom{i+j}{j}=\binom{i+1+u}{u}$, which can be shown by induction on $u$. 
\end{proof}

\begin{lemma}\label{lem:1b}
Let $m>7$ and
$$D_m=\left(\binom{i+1+j}{j}+1\pmod{2}\right)_{2^{m-3}-\frac{14}{2^{4}} 2^{m-3}\leq i< 2^{m-3} ,0\leq j<\frac{1}{2^4} 2^{m-3}}.$$
The relative number of $1$s in $D_m$ is $\left(1-\left(\frac{3}{4}\right)^{m-7}\right)$.
\end{lemma}
\begin{proof} Let $r\geq 1$. In the following we will sometimes use Lucas' Theorem which states the following: Let $i:=\sum_{l=0}^{r-1}i_l2^l$ and $j:=\sum_{l=0}^{r-1}j_l2^l$ with $i_l,j_l\in\{0,1\}$, then $$\binom{i}{j}\equiv \prod_{l=0}^{r-1}\binom{i_l}{j_l}\pmod{2}.$$
First we see from the self-similar structure in $P$, which is a consequence of Lucas' Theorem, that the number of $1$s in 
$$\left(\binom{i+1+j}{j}\mod 2\right)_{0\leq i<2^r,0\leq j<2^r}$$
equals the number of $1$s in 
$$\left(\binom{i+j}{j}\mod 2\right)_{0\leq i<2^r,0\leq j<2^r}.$$
The latter equals the number of $1$s in
$$\left(\binom{i}{j}\mod 2\right)_{0\leq i<2^r,0\leq j<2^r}.$$
This number can be computed as $2^r\cdot 2^r\cdot(3/4)^r$, since $\binom{i}{j}\equiv\prod_{l=0}^{r-1}\binom{i_l}{j_l}\pmod{2}$ and whenever $j_l\leq i_l$ for all $l$ then $\binom{i}{j}\equiv1\pmod{2}$, and $\binom{i}{j}\equiv0\pmod{2}$ else. 
The statement of the Lemma follows then again by the self-similar structure of $D_m$, that is built by a submatrix with $r=m-7$, as described above, stacked 14 times.

\end{proof}

\begin{lemma}\label{lem:1c}
We write $i=8\cdot \tilde{i}+i_0$, where $i_0\in\{0,1,\ldots,7\}$. 
Then 
\begin{align*}
\Big(\binom{i+j}{j}\Big)_{j\geq 0}\cdot\left(\begin{array}{c}\overline{0}\\\overline{1}\\\overline{1}\\\vdots\\\overline{1}\\\overline{1}\\\overline{0}\\\overline{0}\\\vdots\end{array}\right)\equiv\Big(\binom{\tilde{i}+j}{j}\Big)_{j\geq 0}\cdot\left(\begin{array}{c}{0}\\{1}\\1\\\vdots\\{1}\\1\\{0}\\0\\\vdots\end{array}\right)\pmod{2}
\end{align*}
 where the number of $\overline{1}$'s and $1$'s are equal. 
\end{lemma}
\begin{proof}
We observe 
\begin{align*}
\Big(\binom{i+j}{j}\Big)_{j\geq 0}\cdot\left(\begin{array}{c}\overline{0}\\\overline{1}\\\overline{1}\\\vdots\\\overline{1}\\\overline{1}\\\overline{0}\\\overline{0}\\\vdots\end{array}\right)&=\Big(\binom{8\tilde{i}+i_0+j}{j}\Big)_{j\geq 0}\cdot\left(\begin{array}{c}\overline{0}\\\overline{1}\\\overline{1}\\\vdots\\\overline{1}\\\overline{1}\\\overline{0}\\\overline{0}\\\vdots\end{array}\right)\equiv\binom{i_0}{0}\cdot \Big(\binom{8\tilde{i}+8j}{8j}\Big)_{j\geq 0}\cdot\left(\begin{array}{c}{0}\\{1}\\{1}\\\vdots\\{1}\\{1}\\{0}\\0\\\vdots\end{array}\right)\\
&\equiv(\binom{\tilde{i}+j}{j})_{j\geq 0}\cdot\left(\begin{array}{c}{0}\\{1}\\1\\\vdots\\{1}\\1\\{0}\\0\\\vdots\end{array}\right)\pmod{2},
\end{align*}

where we used Lucas' Theorem twice and the fact that $\binom{i_0}{0}=1$ for all $i_0\in\{0,1,\ldots,7\}$. 
\end{proof}

Furthermore, we see:  
\begin{lemma} \label{lem_c}For $k,t$ such that $0\leq k<2^m$, $1\leq t\leq 2^m$ and $k+t\leq 2^m-1$ we have
\begin{enumerate}
	\item $
c_{k+t,t} \equiv \xi_t\cdot A_{k,t}\pmod{2},
$
\item $
d_{k+t,t} \equiv \xi_t\cdot B_{k,t}+c_{k+t,2^m-t}\pmod{2}.
$
\end{enumerate}

\end{lemma}

For the proof of Lemma~\ref{lem_c} we will need the following identity.
\begin{lemma} \label{lem_d}
For all non-negative integers $t, k, l$ we have
$$
\sum^t_{j=0} \binom{t}{j} \cdot \binom{k+l+j}{l} \equiv \binom{k+l}{l-t} \pmod{2}. 
$$
\end{lemma}
\begin{proof}
This is simple induction on $t$.
\end{proof}
From Lemma~\ref{lem_d}, we immediately conclude: 
\begin{corollary} \label{cor_b}
~\
\begin{enumerate}
\item [(a)] $\displaystyle\sum^t_{j=0} \binom{t}{j} \cdot \binom{k+l+j}{l} \equiv 0\pmod{2}$ for $l = 0,1,\ldots, t-1$,
\item [(b)] $\displaystyle\sum^{t-1}_{j=0} \binom{t}{j} \cdot \binom{k+l+j}{l} \equiv \binom{k+l+t}{l}\pmod{2}$ for $l=0,1,\ldots, t-1$,
\item [(c)] $\displaystyle\sum^{t-1}_{j=0} \binom{t}{j} \cdot \binom{k+l+j}{l} \equiv \binom{k+l+t}{l}+\binom{k+l}{l-t}\pmod{2}$ for $l=0,1,2,\ldots$.
\end{enumerate}

\end{corollary}

\begin{proof}[Proof of Lemma~\ref{lem_c}.]
We start with Item (1). We have to show $c_{k+t,t} = \xi_t \cdot A_{k,t}$. This is equivalent to
$$
\binom{k+t+l}{l} = \left(\binom{t}{0}, \ldots, \binom{t}{t-1}\right) \cdot \left(\begin{tabular}{c} $\binom{k+l}{l}$ \\ $\binom{k+1+l}{l}$ \\ $\vdots$ \\ $\binom{k+t-1+l}{l}$ \end{tabular}\right),
$$
in $\mathbb{Z}_2$ for $l = 0, \ldots, t-1$,\\

i.e.,
$$
\sum^{t-1}_{j=0} \binom{t}{j} \cdot \binom{k+l+j}{l} = \binom{k+l+t}{l}
$$
in $\mathbb{Z}_2$ for $l = 0,\ldots, t-1$. This is exactly Corollary~\ref{cor_b} Item (b). 

For the proof of Item (2), i.e. $d_{k+t,t} = \xi_t\cdot B_{k,t}+c_{k+t,2^m-t}$ in $\mathbb{Z}_2$, we see the equivalence with 

$$
\binom{k+t+l}{l} = \left(\binom{t}{0}, \ldots, \binom{t}{t-1}\right) \cdot \left(\begin{tabular}{c} $\binom{k+l}{l}$ \\ $\binom{k+1+l}{l}$ \\ $\vdots$ \\ $\binom{k+t-1+l}{l}$ \end{tabular}\right)+\binom{k+t+l-t}{l-t},
$$
in $\mathbb{Z}_2$ for $l = t, \ldots, 2^m-1$. This is exactly Corollary~\ref{cor_b} Item (c). 
\end{proof}
\begin{remark} \label{rem_a}
{\rm The last entry in $\xi_t$ is $\binom{t}{t-1}=t\equiv 1\mod 2$ if $t$ is odd.}
\end{remark}

From Lemma~\ref{lem_c} we derive the following proposition. 

\begin{proposition} \label{prop_a} 
Let $k\in\mathbb{N}_0,\,t\in\mathbb{N}$ such that $k+t< 2^m$. Then 
\begin{enumerate}
	\item $\kappa_t := c_{k+t,t} \cdot \left(A_{k,t}\right)^{-1}\equiv\xi_t\pmod{2}$ is independent of $k$, and 
\item $\Big(d_{k+t,t}-c_{k+t,t} \cdot \left(A_{k,t}\right)^{-1} B_{k,t}\Big) \cdot\left(\begin{array}{c}\overline{0}\\\overline{1}\\\overline{1}\\\vdots\\\overline{1}\\\overline{1}\\0\\0\\\vdots\\0\end{array}\right)\equiv \binom{\lfloor (k+t)/8\rfloor+1+v}{v}+1  \pmod{2}$, where $v$ is the number of $\overline{1}$s in the vector that consists of $\overline{1}$s, $\overline{0}$s, and $0$s.
\end{enumerate}
\end{proposition}
\begin{proof}
The first item is an immediate consequence of Item (1) in Lemma \ref{lem_c} together with the fact that the square matrix $A_{k,t}$ is regular. 

For the second item note that Item (1) together with Item (2) in Lemma \ref{lem_c} implies 
\begin{eqnarray*}
\left(d_{k+t,t}-c_{k+t,t} \cdot \left(A_{k,t}\right)^{-1} B_{k,t}\right)\cdot \left(\begin{array}{c}\overline{0}\\\overline{1}\\\overline{1}\\\vdots\\\overline{1}\\\overline{1}\\0\\0\\\vdots\\0\end{array}\right)&\equiv&\, (d_{k+t,t}-\xi_t B_{k,t})\cdot \left(\begin{array}{c}\overline{0}\\\overline{1}\\\overline{1}\\\vdots\\\overline{1}\\\overline{1}\\0\\0\\\vdots\\0\end{array}\right)\\
&\equiv&\,c_{k+t,2^m-t}\left(\begin{array}{c}\overline{0}\\\overline{1}\\\overline{1}\\\vdots\\\overline{1}\\\overline{1}\\0\\0\\\vdots\\0\end{array}\right)\\
&\equiv&\,\sum_{i=1}^{v}\binom{8\lfloor (k+t)/8\rfloor+8i+8\{(k+t)/8\}}{8i}\\
&\equiv&\,\sum_{i=1}^{v}\binom{\lfloor (k+t)/8\rfloor+i}{i}\pmod{2}\\
&\equiv&\,\binom{\lfloor (k+t)/8\rfloor+1+v}{v}+1\pmod{2}, 
\end{eqnarray*}
where we applied Lemma \ref{lem:1c} and Lemma \ref{lem:1a}.

\end{proof}

\section{The proof of the Theorem} \label{sect_d}

We will construct now for every $m$ large enough an $N$ with $n_m < N < n_{m+1}$ and an interval $J \subseteq [0,1)$ such that 
$$
\# \left\{0 \leq n < N \left|\left\{2^n \alpha\right\} \in J\right\} \right. \geq N \cdot \lambda (J) + c \cdot \left(\log N\right)^2
$$
(with a fixed absolute positive constant $c$, and where $\lambda (J)$ denotes the length of the interval $J$). This proves the Theorem. \\

For given $m$ (large enough) let $w_l := 2^{m-3}-1-8l$ for $l =0,1,\ldots,  2^{m-7} - 1 =: M.$\\
For $m \geq 7$ all $w_l$ and $M$ are integers. The $w_l$ all are odd, and 
$$
 2^{m-4} < w_M < w_{M-1} < \ldots < w_0 < 2^{m-3} .
$$
Let $N := n_m + 2^m \cdot \left(2^{w_{M}} + 2^{w_{M-1}} + \ldots + 2^{w_0}\right)$.\\


We will consider the sequence elements $x_n := \left\{2^n \alpha\right\}$ for $n=0,1,\ldots, N-1$, i.e., the points $x_0, \ldots, x_{n_{m}-1}$ and the points
$$
x_{n,k} := \{(n_m + 2^m n+k)\alpha\}
$$
for $n=0,1,\ldots, 2^{w_M} + 2^{w_{M-1}} + \ldots + 2^{w_0}-1$ and $k=0,1,\ldots, 2^m-1$. We divide this set of $n$'s in blocks $\mathcal{B}_l$ of the form $n = B_l, B_l +1, \ldots, B_l +2^{w_l} -1$ for $l = 0, \ldots, M$, where $B_0 := 0$ and $B_l := 2^{w_0} + \ldots + 2^{w_{l-1}}$.\\

We construct in the following an interval $J \subseteq [0,1)$ that contains ``too many'' of the points $x_0, \ldots, x_{N-1}$. $J$ will be of the form $J_M \cup J_{M-1} \cup \ldots \cup J_0$ with $J_l := \left[\frac{U(l)}{2^{w_l}}, \frac{V(l)}{2^{w_l}}\right)$, where $0 \leq U(l) < V(l) < 2^{w_l}, \quad U(l)\in\mathbb{N}_0,\, V(l) \in \frac{1}{2} \mathbb{N}_0$ and with $ V(l) - U(l) \in\{\frac{1}{2},\frac{3}{2}\}$, and $\frac{V(l)}{2^{w_l}} = \frac{U(l-1)}{2^{w_{l-1}}}$ for $l = 1,\ldots, M$. That means, $J$ is of the form as sketched in Figure \ref{fig:Figure2}.
\begin{figure}[h]
	\centering
		\includegraphics[width=0.70\textwidth]{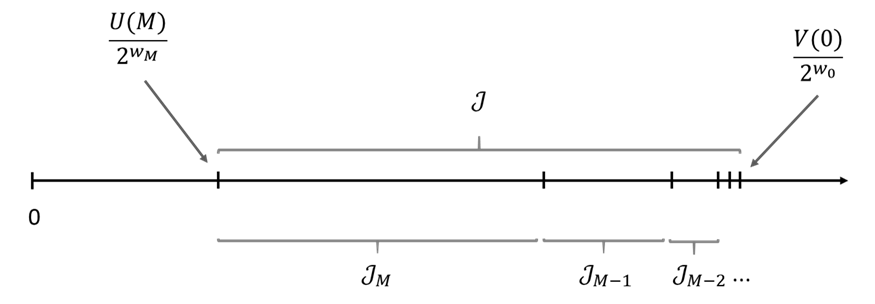}
	\caption{The interval $J$.}\label{fig:Figure2}
\end{figure}

For the length $\lambda(J)$ of the interval $J$ we have
$$
\lambda(J) \leq 2 \cdot \sum^M_{l=0} \frac{1}{2^{w_l}} \leq 4 \cdot \frac{1}{2^{w_M}} < \frac{4}{2^{2^{m-4}}},
$$
where we used $w_M>2^{m-4}$. \\

Now let us recall and use Lemma~\ref{lem_d}:\\
For every $l = 0,\ldots, M$ we consider the points $x_{n,k}$ with $0 \leq k < 2^m$ and $n \in \mathcal{B}_l$. By Lemma~\ref{lem_d} there are at most $2^{m+1}$ integers $c_l^{(i)}$ such that the interval $\left[\frac{c_l^{(i)}}{2^{w_l}}, \frac{c^{(i)}_{l} +1}{2^{w_l}}\right)$ does not contain exactly $2^m$ of these $x_{n,k}$.\\
Altogether there are at most $2^{m+1} \cdot (M+1) < 2^{2m}$ intervals of the form $\left[\frac{c_l^{(i)}}{2^{w_l}}, \frac{c^{(i)}_{l} +1}{2^{w_l}}\right)$ which do not contain exactly $2^m$ of the $x_{n,k}$ with $0 \leq k < 2^m$ and $n \in \mathcal{B}_l$ for some $l = 0,\ldots, M$. So there are at most $2^{2m}$ such ``exceptional intervals'', and the length of the union of these intervals is at most $2^{2m} \cdot \sum^M_{l=0} \frac{1}{2^{w_l}} < \frac{2^{2m+1}}{2^{2^{m-4}}}< \frac{1}{4}$ for $m$ large enough.\\
Hence there exists a sub-interval $Z$ of $[0,1)$ with length at least $\frac{1}{2^{2 m+1}}$ which has empty intersection with every of the exceptional intervals. \\

In the following we construct $J=J_M\cup\cdots\cup J_0$. 

We start with the construction of $J_M$: 

Let $U(M)$ be the least even integer such that $\frac{U(M)}{2^{w_M}} \in Z$. 

For $m$ large enough such $U(M)$ certainly exist. The value $\frac{U(M)}{2^{w_M}}$ is the left border of $J_M$ and hence of $J$. Since $\lambda(J) \leq \frac{4}{2^{2^{m-4}}}$ and $\lambda (Z) \geq \frac{1}{2^{2 m+1}}$, for $m$ large enough we have $J \subseteq Z$ and hence $J$ has empty intersection with every of the exceptional intervals.\\

In the following we construct the right interval boundary $\frac{V_M}{2^{w_M}}$ of $J_M$: For this reason we consider the points $x_{n,k}$ for $n \in \mathcal{B}_M$ and $0 \leq k <    2^m  $. 
Let $\tilde{J}_M := \left[\frac{U(M)}{2^{w_M}}, \frac{U(M) +1}{2^{w_M}}\right)$. \\

We will show now that for each $k\geq 0$ with $ k+w_M <   2^m$ 
there is exactly one $n \in \mathcal{B}_M$ such that $x_{n,k} \in \tilde{J}_M$, and in a second step we will analyze in which sub-interval
$$
J_{M, \gamma} := \left[\frac{U(M)}{2^{w_M}} + \frac{\gamma}{2^{w_M +1}}, \frac{U(M)}{2^{w_M}} + \frac{\gamma+1}{2^{w_M +1}}\right)
$$
for $\gamma = 0,1$ this $x_{n,k}$ is located. \\

Now let $k\geq 0$ with $ k+w_M <  2^m$. We write $U(M) := u_{w_M-1} + u_{w_M-2} \cdot 2 + \ldots + u_0 \cdot 2^{w_M-1}$, 
\begin{eqnarray*}
n&=&e_0(n) + 2 \cdot e_1(n) + \ldots + 2^{w_M-1} \cdot e_{w_M-1}(n)+B_{M}\\
&=& e_0(n) + 2 \cdot e_1(n) + \ldots + 2^{w_M-1} \cdot e_{w_M-1}(n)+2^{w_{M-1}}+2^{w_{M-2}}+\cdots+2^{w_0}\\
&=&e_0(n) + 2 \cdot e_1(n) + \ldots + 2^{w_M-1} \cdot e_{w_M-1}(n)+2^{w_M+8}+2^{w_M+2\cdot 8}\cdots+2^{w_M+M\cdot 8}.
\end{eqnarray*}

Then this leads to the following two systems, where the vector consisting of $\overline{0}$s, $\overline{1}$s, and $0$s, contains $M$ consecutive $\overline{1}$s:
$$
A_{k,w_M} \cdot \left(\begin{tabular}{c} $e_0(n)$ \\ $\vdots$ \\ $e_{w_M-1} (n)$ \end{tabular}\right) + B_{k,w_M}\left(\begin{array}{c}\overline{0}\\\overline{1}\\\vdots\\\overline{1}\\0\\\vdots\\0\end{array}\right) = \left(\begin{tabular}{c} $u_0$ \\ $\vdots$ \\ $u_{w_M -1}$ \end{tabular}\right)
$$
and
$$
c_{k+w_M,w_M} \cdot \left(\begin{tabular}{c} $e_0 (n)$ \\ $\vdots$ \\ $e_{w_M-1}(n)$ \end{tabular}\right) + d_{k+w_M,w_M}\left(\begin{array}{c}\overline{0}\\\overline{1}\\\vdots\\\overline{1}\\0\\\vdots\\0\end{array}\right) = \gamma.
$$
Here $A_{k,w_M}$, $C_{k,w_M}$, $c_{k+w_M,w_M}$, and $d_{k+w_M,w_M}$ are the magnitudes defined in Section ~\ref{sect_c}. Note that here we used the fact that 
$k + w_M<2^m$. Otherwise the system would contain conditions described by using $e_j (n+1)$. 

Since $A_{k,w_M}$ is regular the first system for every $k$ has a unique solution
$$
\left(\begin{tabular}{c} $e_0 (n)$ \\ $\vdots$ \\ $e_{w_M-1}$ \end{tabular}\right) = A_{k,w_M}^{-1} \cdot \left(\begin{tabular}{c} $u_0$ \\ $\vdots$ \\  $u_{w_M-1}$ \end{tabular}\right)-A_{k,w_M}^{-1}B_{k,w_M}\left(\begin{array}{c}\overline{0}\\\overline{1}\\\vdots\\\overline{1}\\0\\\vdots\\0\end{array}\right)
$$
and inserting this solution into the second system leads to
$$
\gamma = c_{k+w_M,w_M} \cdot A_{k,w_M}^{-1} \cdot \left(\begin{tabular}{c} $u_0$ \\ $\vdots$ \\ $u_{w_M-1}$ \end{tabular}\right) -c_{k+w_M,w_M} \cdot A_{k,w_M}^{-1}B_{k,w_M}\left(\begin{array}{c}\overline{0}\\\overline{1}\\\vdots\\\overline{1}\\0\\\vdots\\0\end{array}\right) +  d_{k+w_M,w_M}\left(\begin{array}{c}\overline{0}\\\overline{1}\\\vdots\\\overline{1}\\0\\\vdots\\0\end{array}\right). 
$$
Proposition~\ref{prop_a} Item (1) guarantees that $c_{k+w_M,w_M} \cdot A_k^{-1} = \xi_{w_M}$ is independent of $k$. 
Let $\mathcal{A}_{\delta}(M)$ be the number of $k$s in $\{0,\ldots, 2^m-1-w_M\}$ such that 
$\gamma=\delta$ with fixed $\delta\in\{0,1\}$. 
We define 
$$\mathcal{A}(M):=\max(\mathcal{A}_0(M),\mathcal{A}_1(M)),$$
which we will estimate later.

Altogether we know, that there is a $\gamma\in\{0,1\}$ such that $x_{n,k}\in J_{M,\gamma}$ for at least $\mathcal{A}(M)=:q(M)2^m$ values of $k$s.\\

We distinguish between the cases $\gamma=0$ and $\gamma=1$: 

\textbf{If} $\boldsymbol{\gamma = 0}$, then we choose $J_M := \left[\frac{U(M)}{2^{w_M}} + \frac{U(M) + \frac{1}{2}}{2^{w_M}}\right)$. $J_M$ then contains at least $\mathcal{A}(M)=q(M)2^m$ of the points $x_{n,k}$ with $n \in \mathcal{B}_M$. \\

\textbf{If} $\boldsymbol{\gamma = 1}$, then we choose $J_M := \left[\frac{U(M)}{2^{w_M}} + \frac{U(M) + \frac{3}{2}}{2^{w_M}}\right)$. The reason for this choice is the following: Since $\tilde{J}_M \subseteq Z$ the interval $\tilde{J}_M$ is not an ``exceptional interval'' and contains exactly $2^m$ points $x_{n,k}$ with $0 \leq k < 2^m$ and $n \in \mathcal{B}_M$.\\

The question, how many points of $x_{n,k}$ with $ k+w_M <   2^m$  lie in 
$\left[\frac{U(M) + 1}{2^{w_M}}, \frac{U(M) + 2}{2^{w_M}}\right)$ and, more detailed, in which of the sub-intervals $\left[\frac{U(M) +1}{2^{w_M}} + \frac{\tilde{\gamma}}{2^{w_M+1}}, \frac{U(M) +1}{2^{w_M}} + \frac{\tilde{\gamma}+1}{2^{w_M+1}}\right)$ with $\tilde{\gamma}\in\{0,1\}$ these points are located, now leads to the systems
$$
A_{k,w_M} \cdot \left(\begin{tabular}{c} $e_0(n)$ \\ $\vdots$ \\ $e_{w_M-1} (n)$ \end{tabular}\right) + B_{k,w_M}\left(\begin{array}{c}\overline{0}\\\overline{1}\\\vdots\\\overline{1}\\0\\\vdots\\0\end{array}\right) = \left(\begin{tabular}{c} $u_0$ \\ $\vdots$ \\ $u_{w_M -1}+1$ \end{tabular}\right)
$$
and
$$
c_{k+w_M,w_M} \cdot \left(\begin{tabular}{c} $e_0 (n)$ \\ $\vdots$ \\ $e_{w_M-1}(n)$ \end{tabular}\right) + d_{k+w_M,w_M}\left(\begin{array}{c}\overline{0}\\\overline{1}\\\vdots\\\overline{1}\\0\\\vdots\\0\end{array}\right)  = \tilde{\gamma}.
$$
Hence, as before
$$
\tilde{\gamma} = \underbrace{c_{k+w_M,w_M} \cdot A_{k,w_M}^{-1}}_{=\xi_{w_M}} \cdot \left(\begin{tabular}{c} $u_0$ \\ $\vdots$ \\ $u_{w_M-1}+1$ \end{tabular}\right) -c_{k+w_M,w_M} \cdot A_{k,w_M}^{-1}B_{k,w_M}\left(\begin{array}{c}\overline{0}\\\overline{1}\\\vdots\\\overline{1}\\0\\\vdots\\0\end{array}\right) +  d_{k+w_M,w_M}\left(\begin{array}{c}\overline{0}\\\overline{1}\\\vdots\\\overline{1}\\0\\\vdots\\0\end{array}\right).
$$
Note that we have chosen $U(M)$ to be an even integer, and so $u_{w_M-1}=0$. Moreover, by Remark~\ref{rem_a} the last entry in $\xi_{w_M}$ equals $1$ if $w_M$ is odd (what indeed is satisfied). Therefore $\tilde{\gamma} \equiv \gamma + 1 \equiv 1+1\equiv 0 \pmod{2}$. Hence $\left[\frac{U(M) + 1}{2^{w_M}}, \frac{U(M) + 1}{2^{w_M}} + \frac{1}{2} \cdot \frac{1}{2^{w_M}}\right)$ contains at least $\mathcal{A}(M)=q(M) \cdot 2^m$ points of the $x_{n,k}$ with $k+w_M < 2^m$ and $n \in \mathcal{B}_M$.\\

We summarize: For both choices of $J_M$ we have
$$
\# \left\{n \in \mathcal{B}_M, 0 \leq k < 2^m \left|x_{n,k} \in J_M\right\}\right. \geq 2^m \cdot 2^{w_M} \cdot \lambda \left(J_M\right) + \left(q(M)-\frac{1}{2}\right) \cdot 2^m.
$$

In the next step we will show how to choose the interval $J_{M-1}$. Then it will be clear how we will, quite analogously, choose the intervals $J_{M-2}, \ldots, J_0$.\\

We recall that the interval $J_{M-1}$ is denoted as $J_{M-1} = \left[\frac{U(M-1)}{2^{w_{M-1}}}, \frac{V(M-1)}{2^{w_{M-1}}}\right)$, where $\frac{U(M-1)}{2^{w_{M-1}}} = \frac{V(M)}{2^{w_M}}$. Note that ${w_{M-1}}={w_M}+2^3$. Thus $U(M-1)$ is an even number. Let $U(M-1) := u_{w_{M-1}-1} + u_{w_{M-1}-2} \cdot 2 + \ldots + u_0 \cdot 2^{w_{M-1}-1}$. Similarly to the choice of $J_M$ we will choose $J_{M-1}$ either as
$$
J_{M-1} := J_{M-1}^{(1)} := \left[\frac{U(M-1)}{2^{w_{M-1}}}, \frac{U(M-1)}{2^{w_{M-1}}} + \frac{1}{2} \cdot \frac{1}{2^{w_{M-1}}}\right)
$$
or
$$
J_{M-1} := J_{M-1}^{(2)} := \left[\frac{U(M-1)}{2^{w_{M-1}}}, \frac{U(M-1)}{2^{w_{M-1}}} + \frac{3}{2} \cdot \frac{1}{2^{w_{M-1}}}\right).
$$
To decide, which of the two choices we prefer, we first consider the interval 
$$
\tilde{J}_{M-1} := \left[\frac{U(M-1)}{2^{w_{M-1}}}, \frac{U(M-1)}{2^{w_{M-1}}} + \frac{1}{2^{w_{M-1}}}\right)
$$
and the points $x_{n,k}$ with $n \in \mathcal{B}_{M-1}$ and $k = 0,1,\ldots, 2^m -1$. \\
$\tilde{J}_{M-1}$ by Lemma~\ref{lem_b} contains exactly $2^m$ of these points. Again for $k\geq 0$ such that $ k +w_{M-1}<  2^m$ we ask where exactly these points are located in $\tilde{J}_{M-1}$. Especially, again we ask how these points are distributed to the sub-intervals
$$
\tilde{J}_{M-1, \gamma} := \left[\frac{U(M-1)}{2^{w_{M-1}}} + \frac{\gamma}{2^{w_{M-1}+1}}, \frac{U(M-1)}{2^{w_{M-1}}} + \frac{\gamma+1}{2^{w_{M-1}+1}}\right)
$$
for $\gamma = 0, 1$. Then now for $n = e_0(n) + 2 \cdot e_1(n) + \ldots + 2^{w_{M-1}-1} \cdot e_{w_{M-1}-1}(n)+B_{M-1}$ we arrive at the systems, where the vector consisting of $\overline{0}$,s, $\overline{1}$s, and $0$'s, now contains $(M-1)$ consecutive $\overline{1}$s:

$$
A_{k,w_{M-1}} \cdot \left(\begin{tabular}{c} $e_0(n)$ \\ $\vdots$ \\ $e_{w_{M-1}-1} (n)$ \end{tabular}\right) + B_{k,w_{M-1}}\left(\begin{array}{c}\overline{0}\\\overline{1}\\\vdots\\\overline{1}\\0\\\vdots\\0\end{array}\right) = \left(\begin{tabular}{c} $u_0$ \\ $\vdots$ \\ $u_{w_{M-1} -1}$ \end{tabular}\right)
$$
and
$$
c_{k+w_{M-1},w_{M-1}} \cdot \left(\begin{tabular}{c} $e_0 (n)$ \\ $\vdots$ \\ $e_{w_{M-1}-1}(n)$ \end{tabular}\right) + d_{k+w_{M-1},w_{M-1}}\left(\begin{array}{c}\overline{0}\\\overline{1}\\\vdots\\\overline{1}\\0\\\vdots\\0\end{array}\right)  = \gamma.
$$

Again: Solving this system for $k +w_{M-1} < 2^m$ yields
$$
\gamma=\xi_{w_{M-1}}\begin{pmatrix}u_0\\\vdots\\u_{w_{M-1}-1}\end{pmatrix}-c_{k+w_{M-1},w_{M-1}} \cdot A_{k,w_{M-1}}^{-1}B_{k,w_{M-1}}\left(\begin{array}{c}\overline{0}\\\overline{1}\\\vdots\\\overline{1}\\0\\\vdots\\0\end{array}\right) +  d_{k+w_{M-1},w_{M-1}}\left(\begin{array}{c}\overline{0}\\\overline{1}\\\vdots\\\overline{1}\\0\\\vdots\\0\end{array}\right),
$$
which attains the same value for at least $\mathcal{A}(M-1)=:q(M-1) \cdot 2^m$ $k$s. Here $\mathcal{A}(M-1)$ is defined in the same manner as $\mathcal{A}(M)$. \\

Analogously to the construction of $J_M$ and with the same argumentation we distinguish between the two cases $\gamma = 0$ and $\gamma = 1$. In the first case we choose $J_{M-1} := J_{M-1}^{(1)}$. \\
$J_{M-1}$ then has length $\frac{1}{2} \cdot \frac{1}{2^{w_{M-1}}}$ and contains at least $q(M-1) \cdot 2^m$ of the points $x_{n,k}$ with $k = 0, \ldots, 2^m-1$ and $n \in \mathcal{B}_{M-1}$. In the second case we choose $J_{M-1} := J_{M-1}^{(2)}$. \\
$J_{M-1}$ then has length $\frac{3}{2} \cdot \frac{1}{2^{w_{M-1}}}$ and contains at least $2^m+q(M-1) \cdot 2^m$ of the points $x_{n,k}$ with $k = 0, \ldots, 2^m-1$ and $n \in \mathcal{B}_{M-1}$. \\

In both cases
$$
\# \left\{n \in \mathcal{B}_{M-1}, 0 \leq k < 2^m \left|x_{n,k} \in J_{M-1}\right\} \right. \geq 2^m \cdot 2^{w_{M-1}} \cdot \lambda \left(J_{M-1}\right) + \left(q(M-1)-\frac{1}{2}\right) \cdot 2^m.
$$
Further, since $J_M \cup J_{M-1} \subseteq Z$, we know that the interval 
$$
J_M = \left[\frac{U(M)}{2^{w_M}}, \frac{V(M)}{2^{w_M}}\right) = \left[\frac{U(M) \cdot 2^{i_M-i_{M-1}}}{2^{w_{M-1}}}, \frac{V(M) \cdot 2^{i_M - i_{M-1}}}{2^{w_{M-1}}}\right)
$$
satisfies that $U(M) \cdot 2^{i_M - i_{M-1}}$ and $V(M) \cdot 2^{i_M - i_{M-1}}$ are integers and therefore therefore $J_M$ contains at least $2^m \cdot 2^{w_{M-1}} \cdot \lambda \left(J_M\right)$ of the points $x_{n,k}$ with $k=0, \ldots, 2^m-1$ and $n \in \mathcal{B}_{M-1}$. Together
$$
\# \left\{n \in \mathcal{B}_{M-1}, 0 \leq k < 2^m \left|x_{n,k} \in J_M \cup J_{M-1}\right\}\right. \geq 2^m \cdot 2^{w_{M-1}} \cdot \lambda \left(J_M \cup J_{M-1}\right) + \left(q(M-1)-\frac{1}{2}\right) \cdot 2^m.
$$
In exactly this way we proceed to construct $J_{M-1}, \ldots, J_0$ such that finally for every $l=0, \ldots, M$ we have: 
\begin{equation}\label{equ:number:B_l}
\# \left\{n \in \mathcal{B}_l, 0 \leq k < 2^m \left|x_{n,k} \in J_M \cup \ldots \cup J_l\right.\right\} \geq 2^m \cdot 2^{w_l} \cdot \lambda \left(J_M \cup \ldots \cup J_l\right) + \left(q(l)-\frac{1}{2}\right) \cdot 2^m.
\end{equation}
We set $J := J_M \cup \ldots \cup J_0$. We estimate $\# \left\{0 \leq n < N \left|x_n \in J\right.\right\}$ from below.: 
\begin{eqnarray*}
\# \left\{0 \leq n < N \left|x_n \in J\right.\right\} &=& \# \left\{0 \leq n < n_m\left|x_n \in J\right.\right\} \\
&&+ \sum_{l=0}^{M}\# \left\{n \in \mathcal{B}_l, 0 \leq k < 2^m \left|x_{n,k} \in J_M \cup \ldots \cup J_l\right.\right\}\\
&&+ \sum_{l=0}^{M}\# \left\{n \in \mathcal{B}_l, 0 \leq k < 2^m \left|x_{n,k} \in J_{l-1} \cup \ldots \cup J_0\right.\right\}\\
&\geq &\# \left\{0 \leq n < n_m\left|x_n \in J\right.\right\} \\
&&+ \sum_{l=0}^{M}\# \left\{n \in \mathcal{B}_l, 0 \leq k < 2^m \left|x_{n,k} \in J_M \cup \ldots \cup J_l\right.\right\}\\
&\geq &\# \left\{0 \leq n < n_m\left|x_n \in J\right.\right\}\\
&&+ \sum^M_{l=0} 2^m \cdot 2^{w_l} \cdot \lambda \left(J_M \cup \ldots \cup J_l\right) + \sum^M_{l=0}  \left(q(l)-\frac{1}{2}\right) \cdot 2^m,
\end{eqnarray*}
where in the last step we used \eqref{equ:number:B_l}. 

Now it is the task of estimating $\sum^M_{l=0}  q(l)\cdot 2^m =\sum^M_{l=0}\mathcal{A}(l)$ from below: 
We know that for each $l$ in $\{0,\ldots,M\}$, for at least $\mathcal{A}(l)$ values of $k$s with $k+w_l<2^m$, the term 
\begin{equation}\label{equ:sv}
-c_{k+w_{l},w_{l}} \cdot A_{k,w_{l}}^{-1}B_{k,w_{l}}\left(\begin{array}{c}\overline{0}\\\overline{1}\\\vdots\\\overline{1}\\0\\\vdots\\0\end{array}\right) +  d_{k+w_{l},w_{l}}\left(\begin{array}{c}\overline{0}\\\overline{1}\\\vdots\\\overline{1}\\0\\\vdots\\0\end{array}\right)\pmod{2},
\end{equation}
where here the vector consisting of $\overline{0}$s, $\overline{1}$s, and $0$s contains $l$ consecutive $\overline{1}$s
, takes the same value $0$ or $1$. 
By the second item of Proposition \ref{prop_a} we know that $\mathcal{A}(l)$ can be estimated from below by the number of $k$ with $k+w_l<2^m$ for which $$\binom{\lfloor\frac{k+w_l}{8}\rfloor+1+l}{l}+1\pmod{2}$$
equals $1$. 

Note that $M+1 = 2^{m-7}$. For each $l\in\{0,1,\ldots,M\}$ take now those $k\in\{0,1,\ldots,2^m-1\}$ such that $\lfloor\frac{k+w_l}{8}\rfloor\in\{2^{m-3}-\frac{14}{2^4}2^{m-3},\ldots,2^{m-3}-1\}$. 

Note, that indeed every value of $z$ between $2^{m-3}-\frac{14}{2^4}2^{m-3}$ and $2^{m-3}-1$ is attained by $\lfloor\frac{k+w_l}{8}\rfloor$ for exactly $8$ values of $k$ between $0$ and $2^m-1$. This follows from $\lfloor\frac{8+w_0}{8}\rfloor\leq 2^{m-3}-\frac{14}{2^4}2^{m-3}$ and $\lfloor\frac{2^m-8+w_M}{8}\rfloor\geq 2^{m-3}-1$. 

Hence, 
$$\mathcal{A}(l)\geq 2^3\cdot(\mbox{the number of $1$s in the $l$th column of $D_m$}).$$
Note that $D_m$ is a $(14\cdot 2^{m-7}\times 2^{m-7})$ matrix. From Lemma \ref{lem:1b} we know that $D_m$ contains $14\cdot 2^{m-7}\cdot 2^{m-7}\left(1-\left(\frac{3}{4}\right)^{m-7}\right)$ many $1$s. 
Hence, for $m>7$ large enough such that $\left(1-\left(\frac{3}{4}\right)^{m-7}\right)\geq \frac{31}{32}$, we have:
\begin{align*}
\sum^M_{l=0}\mathcal{A}(l)&\geq 2^3 \cdot (\mbox{the number of $1$s in $D_m$})\\
&= 2^3\left(1-\left(\frac{3}{4}\right)^{m-7}\right) 14\cdot 2^{m-7}2^{m-7}\\
&\geq (M+1)2^m \frac{7}{2^3}\cdot \frac{31}{32}=2^{2m}\frac{7}{2^{10}}\cdot \frac{31}{2^5}.
\end{align*}


By Lemma~\ref{lem_a} we have
$$
\# \left\{0 \leq n < n_m \left|x_n \in J\right.\right\} \geq n_m \cdot \lambda (J) - \delta \cdot \log n_m
$$
with a fixed positive constant $\delta$. 

Further we derive an upper bound for $\lambda \left(J_{l-1} \cup \ldots \cup J_0\right)$:
\begin{equation}\label{equ:bound_J...}
\lambda \left(J_{l-1} \cup \ldots \cup J_0\right)\leq \frac{3}{2} \cdot \left(\frac{1}{2^{w_{l-1}}} + \ldots + \frac{1}{2^{w_0}}\right) \leq 3 \cdot \frac{1}{2^{w_{l-1}}} = 3 \cdot \frac{1}{2^{w_l +2^3}}.
\end{equation}

Altogether, for $m$ large enough (note that $M+1 = 2^{m-7}$) we obtain: 
\begin{eqnarray*}
 \#  \left\{0 \leq n < N \left|x_n \in J\right\}\right. &\geq& n_m \cdot \lambda (J) - \delta \cdot \log n_m + \sum^M_{l=0} 2^m \cdot 2^{w_l} \cdot \lambda(J)\\
&  &  - \sum^M_{l=0} 2^m \cdot 2^{w_l} \cdot 3 \cdot \frac{1}{2^{w_l+2^3}} + \sum_{l=0}^M  \left(q(l)-\frac{1}{2}\right) \cdot 2^m \geq\\
&\geq & N \cdot \lambda(J) - \delta \cdot \log n_m - 2^{2m-7} \cdot \frac{3}{2^{8}} +  \sum_{i=0}^{M} \mathcal{A}(l)-2^{2m-7}\frac{1}{2} \geq \\
& \geq & N \cdot \lambda (J) - \delta \cdot \log N - 2^{2m} \cdot \frac{3+2^7}{2^{15}}
+ 2^{2m}\frac{7}{2^{10}}\frac{31}{2^5} \\
& \geq& N \cdot \lambda (J) - \delta \cdot \log N + 2^{2m} \cdot \frac{86}{2^{15}}  \\
& \geq & N \cdot \lambda (J) + c \cdot (\log N)^2.
\end{eqnarray*}
Here $c$ is a positive constant and we used $2^{2m} \geq \frac{\left(\log N\right)^2}{16}$. \\
The result follows.  
\begin{flushright}
$\Box$
\end{flushright}

\begin{remark}{\rm
The result heavily depends on the strong property of the Pascal matrix $P$ treated in Section~\ref{sect_c}. In searching for normal numbers $\alpha$ with a potentially even better order of normality it makes sense to consider similar constructions as for Levin's binary normal number but with weaker dependence in its generating matrix. Possible candidates could be the examples given in \cite{Beck}.}
\end{remark}

\end{document}